\documentclass[10pt]{amsart}
\input xy
\xyoption{all}
\UseComputerModernTips

\usepackage{array}
\usepackage{longtable}
\usepackage{url}
\usepackage{enumerate}
\usepackage{amsmath}
\usepackage{amssymb}
\usepackage{amsthm}
\usepackage{amsfonts}
  \usepackage{ifthen}
 \newcommand{\mymarginpar}[1]{%
    \marginpar{\ifthenelse{\isodd{\arabic{page}}}{\flushleft
#1}{\flushright #1}}}

 \renewcommand{\phi}{\varphi}

\allowdisplaybreaks[1]

\newcommand{\IC}{\mathbb{C}}

 \newcommand{\IN}{\mathbb{N}}

 \newcommand{\IT}{\mathbb{T}}
 \newcommand{\IZ}{\mathbb{Z}}

\newcommand{\CP}{\mathcal{P}}

\newcommand{\CT}{\mathcal{T}}

\newcommand{\CK}{\mathcal{K}}

\newcommand{\coker}{\text{coker}\,}
\newcommand{\image}{\text{Im}\,}
\newcommand{\supp}{\text{supp}\,}

 \theoremstyle{plain} 
 \newtheorem{Theorem}{Theorem}[section]
 \newtheorem{Lemma}[Theorem]{Lemma}
 \newtheorem{Proposition}[Theorem]{Proposition}

 \theoremstyle{definition} 
 \newtheorem{Definition}[Theorem]{Definition}
 \newtheorem{Remark}[Theorem]{Remark}
 \newtheorem{Example}[Theorem]{Example}

\begin{document}

\title {$K$-theory of $C^*$-algebras of directed graphs}

\author{Menassie Ephrem}
\address{Department of Mathematics and Statistics,
Coastal Carolina University, Conway, SC  29528-6054}
\email{menassie@coastal.edu}

\author{Jack Spielberg}
\address{Department of Mathematics and Statistics,
Arizona State University, Tempe, AZ  85287-1804}
\email{jack.spielberg@asu.edu}



\begin{abstract}
For a directed graph $E$, we compute the $K$-theory of the
$C^*$-algebra $C^*(E)$ from the Cuntz-Krieger generators and
relations.  First we compute the $K$-theory of the crossed product
$C^*(E)\times_\gamma\IT$, and then using duality and the
Pimsner-Voiculescu exact sequence we compute the $K$-theory of
$C^*(E)\otimes\CK \cong  (C^*(E)\times\IT)\times\IZ$.  The method
relies on the decomposition of $C^*(E)$ as an inductive limit of
Toeplitz graph $C^*$-algebras, indexed by the finite subgraphs of
$E$.  The proof and result require no special asssumptions about the
graph, and is given in graph-theoretic terms.  This can be helpful if
the graph is described by pictures rather than by a matrix.
\end{abstract}

\maketitle

\section{introduction}

Since the work of Bratteli in the early 1970's, graphs have been used
as a tool to study a large class of $C^*$-algebras.  Bratteli
classified AF algebras in terms of their diagrams, later called
\textit{Bratteli diagrams} (\cite{bra}).  The current use of directed
graphs in $C^*$-algebras goes back to the work of Cuntz and Krieger
in \cite{cunkri}.  In that work, they associated a $C^*$-algebra to a
finite irreducible 0-1 matrix.

Later, it was noticed that if $A = (a_{ij})$ is an $n \times n$
matrix of 0's and 1's, then $A$ may be viewed as the incidence matrix
of a graph.
It then became natural to view Cuntz-Krieger algebras as arising from
the graphs.  This approach of viewing Cuntz-Krieger algebras as
$C^*$-algebras associated to graphs made the construction more visual
and communicable.

In \cite{kumpasraeren}, Kumjian, Pask, Raeburn and Renault defined
the graph groupoid of a countable row-finite directed graph with no
sinks, and showed that the $C^*$-algebra of this groupoid coincided
with a universal $C^*$-algebra generated by partial isometries
satisfying relations naturally generalizing those given in
\cite{cunkri}.  Since that time, many people have worked on
generalizing these results to arbitrary directed graphs (and beyond
--- for a survey, see \cite{rae}).  In \cite{functorialapproach}, an
approach to the general case is given that results in a direct limit
decomposition of the $C^*$-algebra of a general graph, over the
directed set of its finite subgraphs.  This work motivates the
current paper.

Cuntz and Krieger computed the $K$-theory of their $C^*$-algebra
associated to an irreducible matrix, and showed that it is an
invariant of flow equivalence of the matrix.  Since then several
proofs have been given for the computation of the $K$-theory of the
$C^*$-algebra of a directed graph (\cite{dritom, exelac, kat, pasrae, raeszy, kirchmodels, szy, yi}).  Most of these gave the proof for a restricted class of graphs, e.g. row-finite and/or sourceless (or, in the case of \cite{szy}, for graphs having a finite vertex set).  Proofs of the general case occur in \cite{dritom, kat}.  In this paper we give a proof is simpler than \cite{kat} (that paper treats topological graphs), and does not rely on the row-finite case as does \cite{dritom}.  We follow the general strategy of \cite{pasrae}, first computing
the $K$-theory of the AF algebra $C^*(E)\times_\gamma\IT$, where
$\gamma$ is the gauge action.  We do this  by using the decomposition
of $C^*(E)$ as a direct limit.  Then we give a fairly simple account
of the algebra involved in using the Pimsner-Voiculescu exact
sequence to compute the $K$-theory of $C^*(E)$.  The formula we give for the $K$-theory of the stable AF core $C^*(E)\times_\gamma\IT$ is new, we believe, as is its proof.  (A different formula was given in \cite{pasrae} in the row-finite case.)  We emphasize that in our treatment, no restrictions of any kind are made about row-finiteness, sources and sinks, and cardinality of the graph (the results in \cite{functorialapproach} do not require countability of the vertex and edge sets).

The outline of the paper is as follows.  In section
\ref{preliminaries} we provide the basic definitions of graph
$C^*$-algebras.  In section 3 we compute the $K$-theory of
$C^*(E)\times_\gamma\IT$, and in section 4, that of $C^*(E)$.

The authors wish to thank the referee for his/her detailed and constructive suggestions.

\section{preliminaries}\label{preliminaries}

The paper \cite{functorialapproach} is a reference for the remarks in
this section.  (The survey \cite{rae} is excellent.  However, unlike
that survey, we follow the original convention for graph algebras:
the vertex at the tip of an arrow corresponds to the initial
projection of the partial isometry corresponding to that arrow.)  A
directed graph $E = (E^0, E^1, o, t)$ consists of sets $E^0$ of
\textit{vertices} and $E^1$ of \textit{edges}, and maps $o$, $t : E^1
\to E^0$ identifying the \textit{origin} and \textit{terminus} of an
edge (when an edge is pictured as an arrow between two vertices, the
terminus is the vertex to which it points).  A vertex $x$ is called a
\textit{sink} if $o^{-1}(x) = \emptyset$, a \textit{source} if
$t^{-1}(x) = \emptyset$, and \textit{non-singular} if $o^{-1}(v)$ is
a finite nonempty set.  A \textit{path} is a sequence $e_1e_2 \cdots
e_n$ of edges satisfying $t(e_i) = o(e_{i+1})$ for each $i = 1$,
$\ldots$, $n-1$.  For a path $\mu = e_1 e_2 \cdots e_n$ we define
$o(\mu) = o(e_1)$, $t(\mu) = t(e_n)$, and the \textit{length},
$\ell$, of $\mu$ by $\ell(\mu) = n$.  We regard vertices as paths of
length zero.  Let $E^j$ denote the set of paths of length $j$, and
put $E^* = \cup_{j=0}^\infty E^j$, the \textit{path space} of the
graph.  For $x$, $y \in E^0$, we let $x E^j$, $E^j y$, and $x E^j y$
denote the sets of paths of length $j$ with origin $x$, with terminus
$y$, or both, respectively.

Let $E$ be a directed graph.  A \textit{Cuntz-Krieger $E$-family}
consists of  mutually orthogonal projections $\{s_v : v \in E^0 \}$,
and partial isometries $\{ s_e : e \in E^1\}$, satisfying
\begin{enumerate}
\item $s_{t(e)} = s_e^* s_e$ for all $e \in E^1$.
\item $\sum_{e \in F} s_e s_e^* \le s_v$ for any $v \in E^0$ and
finite subset $F \subseteq v E^1$.
\item $\sum_{e \in v E^1} s_e s_e^* = s_v$ for each non-singular
vertex $v \in E^0$.
\end{enumerate}
The \textit{graph $C^*$-algebra} is the $C^*$-algebra generated by a
universal Cuntz-Krieger $E$-family.  For a path $\mu = e_1 \cdots
e_n$ we write $s_\mu = s_{e_1} \cdots s_{e_n}$.  One easily checks
from the relations that $s_\mu^* s_\mu = s_{t(\mu)}$, $s_\mu s_\mu^*
\le s_{o(\mu)}$ and that $s_\nu^* s_\mu = 0$ unless one of $\mu$,
$\nu$ extends the other.  In this case, e.g. if $\mu = \nu\alpha$, we
have $s_\nu^* s_\mu = s_\alpha$.  Therefore we find that
\[
C^*(E) = \overline{\text{span}} \{s_\mu s_\nu^* : \mu,\;\nu \in E^*
\text{ and } t(\mu) = t(\nu) \}.
\]

Our methods rely crucially on the $C^*$-subalgebras of $C^*(E)$
determined by subgraphs of $E$.  These are termed \textit{relative
Toeplitz graph algebras} in \cite{functorialapproach}, and we
describe them here.  Let $F$ be a subgraph of $E$; that is, $F^0
\subseteq E^0$, $F^1 \subseteq E^1$, and the origin and terminus maps
for $F$ are the restrictions of those for $E$.  We let $S_F$ denote
the set of vertices $v$ of $F$ such that
\begin{enumerate}
\item $v$ is non-singular as a vertex of $E$.
\item $xF^1 = xE^1$.
\end{enumerate}
The \textit{relative Toeplitz Cuntz-Krieger relations} for $F$ and
$S_F$ are the same as the Cuntz-Krieger relations for $F$ except that
(3) is imposed only at vertices in $S_F$.  The \textit{(relative)
Toeplitz graph algebra}, $\CT C^*(F)$, of $F$ is the $C^*$-algebra
universal for the relative Toeplitz Cuntz-Krieger relations.  It is
shown in \cite{functorialapproach} that $\CT C^*(F) \subseteq C^*(E)$
in the obvious way.  (We should indicate the dependence of the
Toeplitz algebra on the choice of subset $S_F \subseteq F^0$, as in
\cite{functorialapproach}; we omit it in this article.)

Given a directed graph $E$, let $\gamma : \IT \to
\textit{Aut}(C^*(E))$ be defined on the generators by $\gamma_z(s_e)
= z s_e$, $e \in E^1$.  (Since  $\{z s_e : e \in E^1\}$ is a
Cuntz-Krieger $E$-family, this does define an automorphism.)  Then we
see that $\gamma_z(s_\mu s_\nu^*) = z^{\ell(\mu) - \ell(\nu)} s_\mu
s_\nu^*$ for any $\mu$, $\nu \in E^*$.  $\gamma$ is called the
\textit{gauge action}, and $(C^*(E),\IT,\gamma)$ is a $C^*$-dynamical
system.  It is a standard fact (see, e.g., \cite{functorialapproach})
that the crossed product algebra $C^*(E) \times_\gamma \IT$ is AF.
In the next section we compute the $K$-theory of $C^*(E)
\times_\gamma \IT$.  In the last section, we use the
Pimsner-Voiculescu exact sequence to compute the $K$-theory of
$C^*(E) \otimes \CK \cong  (C^*(E) \times_\gamma \IT
\times_{\widehat\gamma})$, where $\widehat\gamma$ is the dual action.

For $n \in \IZ$ we let $\zeta_n : \IT \to \IT$ be the $n$th character
of $\IT$:  $\zeta_n(z) = z^n$.  We note some basic computations in
$C^*(E) \times_\gamma \IT$.
First, for $a \in C^*(E)$ we write $\zeta_n a$ for the element of
$C_c(\IT,C^*(E)) \subseteq C^*(E) \times_\gamma \IT$ given by
\[
(\zeta_n a)(z) = \zeta_n(z) a.
\]
Thus $\{ \zeta_n s_\mu s_\nu^* \}$
is a total set in $C^*(E) \times_\gamma \IT$.
We use $\cdot$ for multiplication in $C^*(E)\times_\gamma \IT$.
Thus, if $\mu$, $\nu$, $p$, $q \in E^*$, and $m$, $n \in \IZ$, then
\begin{align*}
\zeta_n s_\mu s_\nu^* \cdot \zeta_m s_p s_q^*
&= \delta_{n,m+\ell(q)-\ell(p)} \zeta_m s_\mu s_\nu^* s_p s_q^* \\
(\zeta_m s_p s_q^*)^*
&= \zeta_{m+\ell(q)-\ell(p)} s_q s_p^*.
\end{align*}
The dual action of $\IZ$ on $C^*(E) \times_\gamma \IT$ is generated
by $\widehat\gamma \in \text{Aut}\,(C^*(E) \times_\gamma \IT)$, where
\[
\widehat{\gamma}(\zeta_m s_p s_q^*)
= \zeta_{m+1} s_p s_q^*.
\]

\section{The $K$-theory of  $C^*(E) \times_\gamma \IT$}

Let $M$ be the incidence matrix of $E$.  Thus
$M:E^0\times E^0\to \IN \cup \{\infty\}$
is defined by requiring that $M(x,y)$ equal the cardinality of
$xE^1y$.  We let $S$ be the set of non-singular vertices of $E$:
$S=\{x\in E^0 : xE^1 \text{ is finite and nonempty } \}$.  For a
subgraph $F$ of $E$ we let $M_F$ denote the incidence matrix of $F$,
and we let $S_F=\{x\in S\cap F^0 : xF^1=xE^1\}$.

\begin{Definition}  We define two maps, $\alpha$ and $\beta$, as
follows.  Let $V=C_c(E^0\times\IZ,\IZ)$ and $W=C_c(S\times\IZ,\IZ)$.
Then $\alpha:V\to V$ is given by
\begin{align*}
(\alpha f)(x,n)&=f(x,n-1),\ f\in V,\\
\intertext{and $\beta:W\to V$ is given by}
(\beta f)(x,n)&=\sum_{y\in S}M(y,x)f(y,n),\ f\in W.\\
\intertext{Equivalently, we may write (for $x\in S$)}
\beta(\delta_{x,n})&=\sum_{e\in xE^1}\delta_{t(e),n}.\\
\end{align*}
\end{Definition}
Thus we may describe $\beta$ loosely by $\beta f=M^t f$.  Note that
$\alpha$ is an isomorphism of $V$, $\alpha(W)=W$, and
$\alpha\circ\beta=\beta\circ\alpha$.
We define $\Phi:V\to K_0(C^*(E)\times\IT)$ by $\Phi(\delta_{x,n}) =
[\zeta_n s_x]$ (this defines $\Phi$ on a basis for $V$, and we extend
to all of $V$ by linearity).  Let $I=(1-\alpha\beta)(W)$.

\begin{Proposition} $\ker(\Phi)=I$,
and $\Phi$ is onto.
\smallskip\noindent
(Thus $K_0(C^*(E)\times\IT)\cong V/I$.)
\end{Proposition}
\begin{proof} First we show the equality.
($\supseteq$):  Let $x\in S$.  For $e\in xE^1$ we have
\begin{flalign*}
&&(\zeta_n s_e^*)^* \cdot \zeta_n s_e^*&=\zeta_n s_es_e^*&\\
&&\zeta_n s_e^* \cdot (\zeta_n s_e^*)^*&=\zeta_{n+1} s_{t(e)},\\
&\text{and hence} &
\sum_{e\in xE^1}\bigl[\zeta_{n+1}s_{t(e)}\bigr]
&=\sum_{e\in xE^1}[\zeta_n s_es_e^*]
=[\zeta_n s_x].\\
&\text{Therefore} &
\Phi\circ(1-\alpha\beta)(\delta_{x,n})
&=[\zeta_n s_x]
-\sum_{e\in xE^1}\bigl[\zeta_{n+1}s_{t(e)}\bigr]
=0.
\end{flalign*}
($\subseteq$):
Let $f\in\ker\Phi$.  Note that $f\in I$ if and only if $\alpha(f)\in
I$.  Also $\Phi(f)=0$ if and only if
$\widehat{\gamma}\bigl(\Phi(f)\bigr)=0$, i.e. if and only if
$\Phi\bigl(\alpha(f)\bigr)=0$.  Thus we may assume that $f(x,i)=0$
whenever $i<0$.  We intend to use this simplification to push
$\Phi(f)$ into $K_0\bigl(C^*(E)^\gamma\bigr)$, since the AF structure
of the fixed-point algebra is easier to deal with than that of
$C^*(E)\times\IT$.  There is one more adjustment necessary for this.

Let $x\in E^0$, $i\ge0$ be such that $f(x,i)\not=0$.  Recall that
$[\zeta_i s_x]=[s_\mu s_\mu^*]$ for any path $\mu\in E^i x$ (for such
a path $\mu$, let $W = s_\mu^* \in C_c(\IT,C^*(E))$; then $W^*W =
s_\mu s_\mu^*$ and $WW^* = \zeta_i s_{t(\mu)}$).  However, if $E$ has
sources, there might not exist such a path.  To get around this
problem, consider a source $y\in E^0$.  Let $D$ be the graph with
$D^0=E^0\cup\{\omega\}$ and $D^1=E^1\cup\{\theta\}$, where
$\omega\not\in E^0$, $o(\theta)=\omega$, and $t(\omega)=y$.  Then
$C^*(E)$ is a full corner in $C^*(D)$, and hence the two algebras
have the same $K$-theory.  (This is easily seen by observing that
$C^*(D)=C^*(E)+s_\theta C^*(E)+C^*(E) s_\theta^*+\IC s_\omega$.)
Moreover, the same observation lets one deduce that $C^*(E)\times\IT$
is a full corner in $C^*(D)\times\IT$ as well.  Thus we may replace
$E$ by $D$ in our situation.  Iterating this process allows us to
assume that for any $(x,i)\in\supp(f)$ there is a path $\mu\in E^*$
such that $[\zeta_i s_x]=[s_\mu s_\mu^*]$.  (The removal of sources
and sinks has a long history in the literature of graph
$C^*$-algebras.  One may add an infinite path leading to a source
rather than just a few edges as we have done here.)

Next we choose a finite dimensional subalgebra of $C^*(E)^\gamma$ to
work in.  Let $F$ be a finite subgraph of $E$ with the following
properties:
\begin{enumerate}
\item $\supp(f)\subseteq F^0\times\IZ$.
\label{itema}
\item For all $(x,i)\in\supp(f)$ there is a path $\mu\in F^i x$.
\label{itemb}
\item $\supp(f)\cap(S\times\IZ)\subseteq S_F\times\IZ$.
\label{itemc}
\item $\sum_{x,i} f(x,i)[\zeta_i s_x]_{K_0\bigl((\CT
C^*(F))\times\IT\bigr)}=0$.
\label{itemd}
\end{enumerate}
(This is possible since
$C^*(E)=\underset{F}{\underrightarrow{\lim}}\,\CT C^*(F)$.)

It follows that the terms of $f$ can be realized within the
fixed-point algebra $(\CT C^*(F))^\gamma$.  Since $(\CT
C^*(F))^\gamma$ is a hereditary subalgebra in the ideal of $(\CT
C^*(F))\times\IT$ that it generates, we may work entirely in $(\CT
C^*(F))^\gamma$.

For $k>0$ we let $C_k(F,S_F)$ denote the finite dimensional
subalgebra of $(\CT C^*(F))^\gamma$ spanned by elements of the form
$s_\mu s_\nu^*$ for which $\ell(\mu)=\ell(\nu)\le k$ (these are
finite dimensional subalgebras whose union is dense in the crossed
product).  Let $k$ be so large that $f(x,i)=0$ whenever $i>k$, and so
that $\sum_{x,i} f(x,i)[\zeta_i s_x]_{K_0(C_k(F,S_F))}=0$.  The
subalgebra $C_k(F,S_F)$ was studied in \cite{kirchmodels}.  In Lemma
4.3 of that paper, all equivalence classes of minimal projections
were described.  We recall that description now.  For $x\in F^0$ put
$\xi_x = s_x-\sum_{e\in xF^1}s_es_e^*$, the \textit{defect
projection} at $x$.  (Note
that $\xi_x=0$ if and only if $x\in S_F$.)  For $y\in F^0$ and $0\le
j<k$ let $N_j(y)=\{s_\mu \xi_y s_\mu^* : \mu\in F^j y\}$, and let
$N_k(y)=\{s_\mu s_\mu^* : \mu\in F^k y\}$.  Then
\[
\bigcup\{N_j(y) : y\in F^0\setminus S_F,\ 0\le
j<k\}\;\cup\;\bigcup\{N_k(y) : y\in F^0\}
\]
is a family of pairwise orthogonal minimal projections in
$C_k(F,S_F)$ with sum 1.  Moreover, two such projections $p\in
N_j(y)$ and $q\in N_i(w)$ are equivalent if and only if $j=i$ and
$y=w$.   It follows (see, e.g., the last part of the proof of Lemma
4.3 of \cite{kirchmodels}) that for any path $\mu\in F^*$ with
$\ell(\mu)\le k$ we have
\[
s_\mu s_\mu^* = \sum_{j<k-\ell(\mu)}\;\sum_{\nu\in t(\mu)F^j}
s_{\mu\nu}\xi_{t(\nu)} s_{\mu\nu}^* + \sum_{\nu\in
t(\mu)F^{k-\ell(\mu)}} s_{\mu\nu} s_{\mu\nu}^*.
\]
Hence, replacing the sum on $\nu$ by the sum on $y = t(\nu)$,
\begin{align*}
[\zeta_{\ell(\mu)} s_{t(\mu)}] = [s_\mu s_\mu^*]
&=\sum_{j=0}^{k-\ell(\mu)-1} \; \sum_{y\in F^0} M_F^j
\bigl(t(\mu),y\bigr) [\zeta_{\ell(\mu)+j} \xi_y]
+ \sum_{y\in F^0} M_F^{k-\ell(\mu)}
\bigl(t(\mu),y\bigr) [\zeta_k s_y] \\
&= \sum_{j=\ell(\mu)}^{k-1} \; \sum_{y\in F^0}
M_F^{j-\ell(\mu)}\bigl(t(\mu),y\bigr) [\zeta_j \xi_y]
+ \sum_{y\in F^0} M_F^{k-\ell(\mu)}\bigl(t(\mu),y\bigr)
[\zeta_k s_y]
\end{align*}
Thus
\begin{align*}
0 &= \Phi(f) \\
&= \sum_{i=0}^k \sum_{x\in F^0} f(x,i)[\zeta_i s_x] \\
&= \sum_{i=0}^k \sum_{x\in F^0} f(x,i)
\left(\sum_{j=i}^{k-1} \sum_{y\in F^0} M_F^{j-i}(x,y)[\zeta_j \xi_y]
+ \sum_{y\in F^0} M_F^{k-i}(x,y) [\zeta_k s_y]\right) \\
&= \sum_{j=0}^{k-1} \sum_{y\in F^0}
\left(\sum_{x\in F^0} \sum_{i=0}^j
M_F^{j-i}(x,y)f(x,i)\right)[\zeta_j \xi_y]
+ \sum_{y\in F^0}\left(\sum_{x\in F^0} \sum_{i=0}^k
M_F^{k-i}(x,y)f(x,i)\right) [\zeta_k s_y].
\end{align*}
We pause in the proof to introduce some definitions.
\begin{Definition}
For $g \in V$ let $g_i\in C_c(E^0,\IZ)$ be defined by
$g_i(x)=g(x,i)$, and let $A=M_F^t$.
\end{Definition}
Now we may write:
\[
\sum_{x\in F^0}M_F^r(x,y)f(x,i)=(A^r f_i)(y).
\]
We then have
\begin{enumerate}\setcounter{enumi}{4}
\item $\sum_{i=0}^j (A^{j-i}f_i)(y)=0$, for all $y\in F^0\setminus
S_F$ and $0\le j<k$,
\label{iteme}
\item $\sum_{i=0}^k (A^{k-i}f_i)(y)=0$, for all $y\in F^0$.
\label{itemf}
\end{enumerate}
We pause the proof once more to introduce new notation.
\begin{Definition}\label{betazero}
Let $V_0=C_c(E^0,\IZ)$, $W_0=C_c(S,\IZ)$, and let  $\beta_0:W_0\to
V_0$  be defined by $\beta_0 = \beta \bigr|_{W_0}$.  (We may also describe $\beta_0$ by
$\beta_0(\delta_x) = \sum_{e \in xE^1} \delta_{t(e)}$,
$x \in S$ (compare \cite{kat}, Proposition 6.11.)
\end{Definition}
Then $\beta_0$ agrees with $A$ on $C_c(S_F,\IZ)$.
Using \eqref{iteme}, \eqref{itemf}, \eqref{itemc}, and \eqref{itema},
we find that
\begin{align*}
f_0&\in W_0,\\
f_1+Af_0&\in W_0,\\
\cdots&\\
f_{k-1}+A(f_{k-2}+A(\cdots+A(f_1+Af_0)))&\in W_0,\\
f_k+A(f_{k-1}+A(\cdots+Af_0))&=0.
\end{align*}
Thus $A$ can be replaced by $\beta_0$ in these formulas.  Let us
define $h\in W$ by
\[
h_i=
\begin{cases}
f_0,&\text{ if }i=0\\
f_i+\beta_0 h_{i-1},&\text{ if }0<i<k\\
0,&\text{ if }i<0\text{ or }i\ge k.
\end{cases}
\]
Then it is immediate that $(1-\alpha\beta)h=f$, proving that $f\in I$.

Finally, we show that $\Phi$ is onto.  We have already seen that the
classes of minimal projections in $C_k(F,S_F)$ are in the range of
$\Phi$.  Since the images of these under $\widehat{\gamma}_*=\alpha$
generate $K_0(C^*(E)\times\IT)$, it follows that $\Phi$ is onto.
\end{proof}

\section{The $K$-theory of $C^*(E)$}

The rest of our argument consists of algebraic manipulations.  We
first give some notation.

\begin{Definition}
Define maps $e_i:V\to V$ by
\[
e_i(f)_j=
\begin{cases}
f_i,&\text{if $j=i$} \\
0,&\text{if $j\not=i$}.
\end{cases}
\]
Let $q_i:V\to V$ be defined by $q_i=\sum_{j\le i}e_j$ (note that the
sum is finite on elements of $V$).  We note that $e_i$ and $q_i$
commute with $\beta$, and that $\alpha^j\circ
e_i=e_{i+j}\circ\alpha^j$ (and similarly for $q_j$).  We define
$E:V\to V_0$ by $E(f)=\sum_i f_i$, and $\phi:V_0\to V$ by
\[
\phi(x)_j=
\begin{cases}
x,&\text{if $j=0$} \\
0,&\text{if $j\not=0$}.
\end{cases}
\]
Then $E\circ\alpha=E$, $E\circ\beta=\beta_0\circ E$ and
$\phi\circ\beta_0=\beta\circ\phi$.
\end{Definition}

\begin{Lemma}\label{lemma_1}
Let $g\in V$ and $h\in W$ be such that
\begin{equation}
\label{eqn_1_a}
(1-\alpha^{-1})g=(1-\alpha\beta)h.
\end{equation}
Then $E(h)\in\ker(1-\beta_0)$, and
$g+\phi\circ E(h)\in I$.
\end{Lemma}
\begin{proof}
Applying $E$ to equation \eqref{eqn_1_a} gives
$0=E\circ(1-\alpha\beta)h=(1-\beta_0)\bigl(E(h)\bigr)$.
Next, applying $e_i$ to equation \eqref{eqn_1_a} gives
\begin{align}
e_i(g)-\alpha^{-1} e_{i+1}(g) &= e_i(h)-\alpha\beta e_{i-1}(h)
\notag \\
&=e_i(h)-e_{i-1}(h) + (1-\alpha\beta)e_{i-1}(h). \label{eqn_1_b}
\end{align}
Adding equations \eqref{eqn_1_b} for $i\le j$ gives
\begin{equation}\label{eqn_1_c}
q_j(g)-\alpha^{-1} q_{j+1}(g) = e_j(h)+(1-\alpha\beta)q_{j-1}(h).
\end{equation}
Applying $\alpha^{-j}$ to equation \eqref{eqn_1_c} gives
\begin{equation}\label{eqn_1_d}
\alpha^{-j}q_j(g)-\alpha^{-(j+1)}q_{j+1}(g) =
\alpha^{-j}e_j(h)+(1-\alpha\beta)\alpha^{-j}q_{j-1}(h).
\end{equation}
For $m<n$, we add equations \eqref{eqn_1_d} for $m\le j < n$ to get
\begin{equation}\label{eqn_1_e}
\alpha^{-m}q_m(g) - \alpha^{-n}q_n(g)
= \sum_{j=m}^{n-1}\alpha^{-j}e_j(h) +
(1-\alpha\beta)\sum_{j=m}^{n-1}\alpha^{-j}q_{j-1}(h).
\end{equation}
Choose $m$ and $n$ so that $g_i=h_i=0$ for $i \le m$ and $i \ge n$.
Then
$q_n(g)=g$, $q_m(g)=0$, and
$\sum_{j=m}^{n-1}\alpha^{-j}e_j(h)=\phi\bigl(E(h)\bigr)$.  Thus from
equation \eqref{eqn_1_e} we obtain
\begin{equation}\label{eqn_1_f}
g+\alpha^n\circ\phi\circ E(h)\in I.
\end{equation}
But for any $j$ we have
\begin{align*}
\alpha^j&\circ\phi\circ E(h)-\alpha^{j+1}\circ\phi\circ E(h)
=\alpha^j\circ\phi\circ E(h)-\alpha^{j+1}\circ\phi
\circ\beta_0\circ
E(h),
\text{ since } E(h) \in \ker(1-\beta_0), \\
&=\alpha^j\circ\phi\circ E(h)-\alpha^{j+1}\circ\beta
\circ\phi\circ E(h)
=(1-\alpha\beta)\circ\alpha^j\circ\phi\circ E(h)
\in I.
\end{align*}
From this and equation \eqref{eqn_1_f} we have that $g+\phi\circ
E(h)\in I$.
\end{proof}

\begin{Lemma}\label{lemma_2}
$1-\phi\circ E=(1-\alpha^{-1})\Bigl(-\sum_{j<0}\alpha^{-j}q_j +
\sum_{j\ge0}\alpha^{-j}(1-q_j)\Bigr)$.
\end{Lemma}
\begin{proof}
This is a straightforward computation.  We have
\[
\sum_{j<0}\alpha^{-j}q_j
=\sum_{j<0}\sum_{i\le j}\alpha^{-j}e_i\\
=\sum_{i<0}\sum_{j=i}^{-1}\alpha^{-j}e_i\\
=\sum_{i<0}\sum_{j=1}^{-i}\alpha^{j}e_i.
\]
Hence
\begin{equation}\label{lemma_2_a}
(1-\alpha^{-1})\sum_{j<0}\alpha^{-j}q_j
=\sum_{i<0}(\alpha^{-i} - 1)e_i.
\end{equation}
Similarly, we have
\[
\sum_{j\ge0}\alpha^{-j}(1-q_j)
=\sum_{j\ge0}\sum_{i>j}\alpha^{-j}e_i\\
=\sum_{i>0}\sum_{j=0}^{i-1}\alpha^{-j}e_i.
\]
Hence
\begin{equation}\label{lemma_2_b}
(1-\alpha^{-1})\sum_{j\ge0}\alpha^{-j}(1-q_j)
=\sum_{i>0}(1 - \alpha^{-i})e_i.
\end{equation}
Finally, combining equations \eqref{lemma_2_a} and \eqref{lemma_2_b},
we find that in the statement of the lemma, the right-hand side of
the equation equals
\[
\sum_{i\not=0}e_i-\sum_{i\not=0}\alpha^{-i}e_i
=\sum_i e_i-\sum_i \alpha^{-i}e_i \\
=1-\phi\circ E.\qedhere
\]
\end{proof}

Now we will compute the $K$-theory of $C^*(E)$.  Let
$\widetilde{V}=V/I$.  Since
$\alpha(I)=I$, $\alpha$ descends to an automorphism
$\widetilde\alpha$ of $\widetilde{V}$.  Under the isomorphism of
$K_0(C^*(E)\times\IT)$ with $\widetilde{V}$,
$\widetilde\alpha$ corresponds to the dual action of $\IZ$.  So by
the
Pimsner-Voiculescu exact sequence, we must identify the kernel and
cokernel of $1-\widetilde\alpha^{-1}$.  We will show that the kernel
and
cokernel of $1-\widetilde\alpha^{-1}$ are isomorphic to those of
$1-\beta_0$.  We let $\widetilde\phi$ denote the composition of
$\phi$ with the quotient map of $V$ onto $\widetilde{V}$.

\begin{Proposition}\label{proposition}
With the above notation, the kernel and cokernel of
$1-\widetilde\alpha^{-1}$ are isomorphic to those of $1-\beta_0$.
\end{Proposition}
\begin{proof}
We have for $x\in W_0$,
\begin{align*}
(1-\widetilde\alpha^{-1})\circ\widetilde\phi(x)
&=(1-\alpha^{-1})\circ\phi(x) + I \\
\widetilde\phi\circ(1-\beta_0)(x)
&=(1-\beta)\circ\phi(x) + I. \\
\intertext{Since}
(1-\beta)\circ\phi(x) - (1-\alpha^{-1})\circ\phi(x)
&=(\alpha^{-1}-\beta)\circ\phi(x)
 =(1-\alpha\beta)\circ\alpha^{-1}\circ\phi(x)\in I, \\
\intertext{we find that (on $W_0$)}
\widetilde\phi\circ(1-\beta_0)
&=(1-\widetilde\alpha^{-1})\circ\widetilde\phi.
\end{align*}
Therefore $\widetilde\phi$ defines maps:
$\ker(1-\beta_0) \to \ker(1-\widetilde\alpha^{-1})$ and
$\coker(1-\beta_0) \to \coker(1-\widetilde\alpha^{-1})$.  We will
show that these maps are isomorphisms.

First we treat the map on kernels.  For surjectivity, let
$g+I\in\ker(1-\widetilde\alpha)$.  Then $(1-\alpha^{-1})g\in I$.
Thus there is $h\in W$ such that
$(1-\alpha^{-1})g=(1-\alpha\beta)h$.  By
Lemma \ref{lemma_1} we have that $E(h)\in\ker(1-\beta_0)$ and that
$g+I=-\phi\bigl(E(h)\bigr)+I=\widetilde\phi\bigl(E(-h)\bigr)
\in\widetilde\phi\bigl(\ker(1-\beta_0)\bigr)$.  For injectivity, let
$x\in\ker(1-\beta_0)$ and suppose that $\widetilde\phi(x)=0$.  Then
$\phi(x)=(1-\alpha\beta)h$ for some $h\in W$.  Applying $e_i$ to this
equation gives
\[
h_i-\beta_0(h_{i-1})=
\begin{cases}
x,&\text{if $i=0$} \\ 0,&\text{if $i\not=0$}.
\end{cases}
\]
Thus for any $i<0$, if $h_i\not=0$ then $h_{i-1}\not=0$.  Since $h$
is finitely non-zero, we must have $h_i=0$ for $i<0$.  Then $h_0=x$,
and for $i>0$ we have $h_i=\beta_0(h_{i-1})=\cdots=\beta_0^i(h_0)
=\beta_0^i(x)=x$.  Again since $h$ is finitely non-zero, we must have
$x=0$.

We now treat the map on cokernels.  For injectivity, let $x\in V_0$
be such that
$\widetilde\phi(x) \in (1-\widetilde\alpha^{-1})(V)$.  Thus
$\phi(x)\in(1-\alpha^{-1})(V) +I$.  Then there are $g\in V$ and $h\in
W$
such that $\phi(x)=(1-\alpha^{-1})g+(1-\alpha\beta)h$.  Applying $E$
gives
$x=0+(1-\beta_0)\bigl(E(h)\bigr)\in\image(1-\beta_0)$.  Finally, for
surjectivity, let $g\in V$.  Let $x=E(g)\in V_0$.
Then by Lemma \ref{lemma_2} we
know that $g\in\phi(x)+(1-\alpha^{-1})(V)$, so
$g+I\in\widetilde\phi(x)+(1-\widetilde\alpha^{-1})(\widetilde{V})$.
\end{proof}

\begin{Theorem}
Let $E$ be a directed graph.  Recall the map $\beta_0$ from
Definition \ref{betazero}.  Then $K_1(C^*(E))\cong \ker(1-\beta_0)$
and $K_0(C^*(E))\cong \coker(1-\beta_0)$.
\end{Theorem}
\begin{proof}
This follows from Proposition \ref{proposition}, and the remarks
before it.
\end{proof}

\begin{Remark}
\label{computing}
We remark that the $K$-groups of $C^*(E)$ may be described as follows, using the alternate description of Definition \ref{betazero}:
\begin{align*}
K_1(C^*(E)) &= \bigl\{ f \in C_c(S,\IZ) :
f(x) = \sum_{e \in E^1 x} f\bigl(o(e)\bigr) \bigr\} \\
K_0(C^*(E)) &= C_c(E^0,\IZ) /
\text{span}\,\bigl\{ \delta_x - \sum_{e \in xE^1} \delta_{t(e)} : x \in E^0 \bigr\}
\end{align*}
\end{Remark}
We mention that when the graph is described by pictures rather than by a matrix, the formulas given Remark \ref{computing} (also in \cite{kat}) can be easier to apply than those giving the $K$-groups as the kernel and the cokernel of a matrix.

\begin{Example}
Consider the following graph $E$:

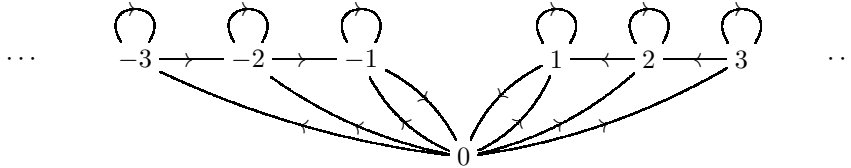
\[
\xymatrix{
 {\bf \cdots}
 &
 -3
   \ar@{-}@(ur,ul)[]|-{\object@{<}}
   \ar@{-}[r]|-{\object@{>}}
  &
 -2
   \ar@{-}@(ur,ul)[]|-{\object@{<}}
   \ar@{-}[r]|-{\object@{>}}
  &
 -1
   \ar@{-}@(ur,ul)[]|-{\object@{<}}
   \ar@{-}@/^/[dr]|-{\object@{>}}
  &&
 1
   \ar@{-}@(ur,ul)[]|-{\object@{<}}
   \ar@{-}@/_/[dl]|-{\object@{>}}
  &
 2
   \ar@{-}@(ur,ul)[]|-{\object@{<}}
   \ar@{-}[l]|-{\object@{>}}
  &
 3
   \ar@{-}@(ur,ul)[]|-{\object@{<}}
   \ar@{-}[l]|-{\object@{>}}
  &
    {\bf \cdots}
\\
 &&&&
 0
   \ar@{-}@/_/[ur]|-{\object@{>}}
   \ar@{-}@/_/[urr]|-{\object@{>}}
   \ar@{-}@/_/[urrr]|-{\object@{>}}
   \ar@{-}@/^/[ul]|-{\object@{>}}
   \ar@{-}@/^/[ull]|-{\object@{>}}
   \ar@{-}@/^/[ulll]|-{\object@{>}}
}
\]

Since the graph is transitive and has loops, its $C^*$-algebra is a (UCT) Kirchberg algebra (\cite{functorialapproach}).  Since there are infinitely many vertices, $C^*(E)$ is non-unital, and hence is stable, by a theorem of Zhang (\cite{zha}).  We use the Remark \ref{computing} to compute the $K$-theory of $C^*(E)$.

First note that $S = \IZ \setminus \{0\}$.  Now let $f \in K_1(C^*(E))$.  Thus $f \in C_c(S,\IZ) \subseteq C_c(E^0,\IZ)$.  In particular, $f(0)=0$.  For $n \ge 1$, we have $f(n) = f(n) + f(n+1) + f(0)$, and hence $f(n+1)=0$.  Similarly, $f(n-1)=0$ for $n \le -1$.  Finally $0 = f(0) = f(-1) + f(1)$, so that $f(-1) = -f(1)$.  We see that $K_1(C^*(E)) = \IZ(\delta_1 - \delta_{-1}) \cong \IZ$.

Now,  let $[f]$ denote the class in $K_0$ of an element $f \in C_c(E^0,\IZ)$.  For $n \ge 1$ we have
$[\delta_n] = [\delta_n] + [\delta_{n-1}]$, and hence $[\delta_{n-1}] = 0$.  Similarly, $[\delta_{n+1}] = 0$ for all $n \le -1$.  Thus $[\delta_n] = 0$ for all $n \in \IZ$.  It follows that $K_0(C^*(E)) = 0$.

The algebra $C^*(E)$ is the stable form of the Kirchberg algebra denoted $\CP_\infty$ by Blackadar (\cite{bla2}).
\end{Example}

Additional interesting examples appear in
\cite{kirchautos}.  In the examples computed there, the graphs are more easily presented (and understood) by diagrams rather than by matrices.

\end{document}